\documentclass[11pt]{amsart}
\usepackage{amssymb,latexsym}

\newcommand{\C}{\mathbb{C}}

\newcommand{\R}{\mathbb{R}}
\newcommand{\Z}{\mathbb{Z}}

\newcommand{\PP}{\mathbb{P}}

\def\vs{\vspace{2mm}}
\def\Hom{{\rm Hom\/}}

\def\p{\partial}

\def\be{\begin{equation}}
\def\ee{\end{equation}}

\newtheorem{theorem}{Theorem}[section]
\newtheorem{prop}[theorem]{Proposition}
\newtheorem{lemma}[theorem]{Lemma}
\newtheorem{remark}[theorem]{Remark}

\newtheorem{definition}[theorem]{Definition}
\newtheorem{cor}[theorem]{Corollary}
\newtheorem{example}[theorem]{Example}
\newtheorem{conj}[theorem]{Conjecture}
\DeclareMathOperator{\ep}{\epsilon}

\begin{document}

\title{Minimal genus for 4-manifolds with $b^+=1$}

\author{Bo Dai \& Chung-I Ho \& Tian-Jun Li}
\address{LMAM, School of Mathematical Sciences \\ Peking
University \\ Beijing 100871, P. R. China}
\email{daibo@math.pku.edu.cn}
\address{Mathematics Department\\  National Tsing Hua University\\ Hsinchu, Taiwan}
\email{ciho@math.cts.nthu.edu.tw}
\address{School  of Mathematics\\  University of Minnesota\\ Minneapolis, MN 55455}
\email{tjli@math.umn.edu}

\begin{abstract}
We derive an adjunction inequality for 
 any smooth, closed, connected, oriented 4-manifold $X$ with $b^+=1$.  
 This inequality  depends only on the cohomology algebra and 
generalizes the inequality of Strle in the case of $b_1=0$. We demonstrate that
the inequality is especially powerful when $2\tilde \chi+3\sigma\geq 0$, where $\tilde \chi$ is the modified Euler number taking account of the cup product on $H^1(X;\Z)$. 
\end{abstract}
\maketitle

\tableofcontents

\section{Introduction}

Let $X$ be a smooth, closed, connected, oriented 4-manifold. 
\begin{definition}
For any class $A\in H_2(X;\Z)$/Tor, let
 $$ mg_X(A)=\min \{ g(\Sigma)|\Sigma\ \text{is a connected, smoothly embedded surface with} \ [\Sigma]=A \}. $$
\end{definition}

The minimal genus function $mg_X$ is a basic invariant of $X$. For rich history on this function  see the excellent surveys of Lawson, \cite{Lawson}, \cite{Lawson-sphere}. 
For manifolds with $b^+(X)=1$, in this paper we will provide a general adjunction type  bound for  $mg_X$.  Unlike the well known adjunction  bounds in \cite{KM}, \cite{MST96}, \cite{OS}, 
our bound only depends on the cohomology algebra of $X$.

To state our bound we need to introduce the modified Euler characteristic $\tilde \chi$. 
Let  $\Lambda=\oplus _{i=0}^4\Lambda^i$ be a cohomology algebra of $b^+=1$ type.  We refer the readers to  Section~\ref{algebra} for precise definition.
Consider  the skew-symmetric bilinear form
\begin{equation}
T:\Lambda^1\times \Lambda^1\to  \Lambda^2.
\end{equation}
Let $\tilde b_1(\Lambda)$ be the rank of $T$. $\tilde b_1(\Lambda)$ is always an even number and the image of $T$ is at most one dimensional (see Subsection \ref{b+=1algebra}). 
Let
$\tilde \chi(\Lambda)=2+b_2(\Lambda)-2\tilde b_1(\Lambda)$, where $b_i(\Lambda)$ is the rank of $\Lambda^i$. Let $\sigma(\Lambda)$ be the signature of the symmetric bilinear form $\Gamma: \Lambda^2 \times \Lambda^2 \to \Lambda^4 \cong \Z$.

\begin{definition}

Let $\Lambda$ be a cohomology  algebra of $b^+=1$ type. 
A class $c\in \Lambda^2$ is called an adjunction class if  it is characteristic and either of the following conditions is satisfied, 

\begin{itemize}

\item[(I)]  $c\cdot c > \sigma(\Lambda)$,


\item[(II)]  $c \cdot c \geq 2\tilde \chi(\Lambda)+3\sigma(\Lambda)$ and $c$ pairs  non-trivially with  ${\rm Im}\, T$ when $T$ is non-trivial.

\end{itemize}

\end{definition}

We make a few remarks for the two types of adjunction classes. A class $c$ is characteristic if $c\cdot A \equiv A\cdot A \ (\text{mod}\ 2)$ for all $A\in \Lambda^2$. The conditions for adjunction classes come from  dimension formula and  wall crossing formulae in Seiberg-Witten theory for 4-manifolds with $b^+=1$ as explained in Section 3. 
If $T$ is trivial, there are only type I adjunction classes. A well-known result says that characteristic classes satisfy $c\cdot c \equiv \sigma(\Lambda)\ (\text{mod}\ 8)$. So $c\cdot c > \sigma(\Lambda)$ is equivalent to 
 \begin{equation} \label{typeI} c\cdot c \geq \sigma(\Lambda)+8.
 \end{equation} By direct computation, we also have 
\begin{equation}\label{typeII}
 2\tilde \chi(\Lambda)+3\sigma(\Lambda) =\sigma(\Lambda) +8 -4 \tilde b_1 (\Lambda).
 \end{equation}
  When $T$ is non-trivial,  $\tilde b_1 (\Lambda)$ is at least $ 2$.
 Thus it follows from  \eqref{typeI} and \eqref{typeII} that  
   it is  possible that  some type II adjunction classes  are not of type I, and a 
 type I adjunction class is also of type II if and only if it pairs non-trivially with Im$T$.

\begin{definition}
Let $A\in \Lambda^2$. 
For any  class $c$ of adjunction type, the $c$-genus of $A$ is defined as 

\begin{equation}\label{c-genus}
h_c(A)=\begin{cases}   1+ \frac{A\cdot A-|c\cdot A|}{2}   &\mbox{if } A \ne 0 \\
0 &\mbox{if  } A =0.
\end{cases}
\end{equation}
Let 
$h(A)=\max h_c(A)$, where the maximum is taken among all adjunction classes of $\Lambda$. 

\end{definition}

Notice that $h_c(A)$ is always an  integer since $c$ is characteristic. 
Notice also  that 
$h_c(A)=h_c(-A)$, hence $h(A)=h(-A)$.  
We remark that the minimal genus function $mg_X$ also has
this symmetry. 

More generally,  $h$ is invariant under the group of  automorphisms of $\Lambda^2$ that preserve the bilinear form $\Gamma$ and Im$T$.
This property will be used to  calculate  $h$  explicitly  when  $2\tilde \chi+3\sigma\geq 0$.

For $X$ a smooth, closed, connected, oriented 4-manifold with $b^+=1$, let $\Lambda(X)=H^*(X;\Z)$/Tor. Then $\Lambda(X)$ is a cohomology algebra of $b^+=1$ type. For any cohomology class $A\in H^2(X;\Z)$/Tor, let $mg_X(A) =mg_X(\text{PD}(A))$, where PD$(A)\in H_2(X;\Z)$/Tor is the Poincar\'e dual of $A$. For convenience, we will identify $H^2(X;\Z)$ with $H_2(X;\Z)$ by Poincar\'e duality whenever necessary. 
Our  bound of the minimal genus function is 

\begin{theorem}\label{main}
Let $X$ be a smooth, closed, connected, oriented 4-manifold with $b^+(X)=1$. 
Then for any $A\in \Lambda^2(X)$ with $A\cdot A\geq 0$, we have
$$mg_X(A) \geq  h(A).$$
\end{theorem}

Notice that the bound of $mg_X$ in Theorem   \ref{main} does not involve any Seiberg-Witten invariant. 
This  inequality is motivated by the one given by Strle in \cite{S} for classes with positive square in 4-manifolds with $b^+=1$ and $b_1=0$, 
which is proved via 
$L^2$ moduli spaces of manifolds with cylindrical ends. 
Our proof uses the wall crossing formulae for Seiberg-Witten invariants.


When $2\tilde \chi(\Lambda)+3\sigma(\Lambda)\geq 0$,  $h$ is easier to calculate and
offers a sharp bound in the following sense.


\begin{theorem}\label{optimal} Let $\Lambda$ be a cohomology algebra of $b^+=1$ type, and  
suppose $2\tilde \chi(\Lambda)+3\sigma(\Lambda)\geq 0$.  Then $h(A)\geq 0$ for any $A\in \Lambda^2$ with 
$A\cdot A >0$, or $A\cdot A=0$ and $A$ is primitive.

Moreover,  there exists a smooth,  closed, connected, oriented 4-manifold $X_{\Lambda}$ 
and an isomorphism $H^*(X_{\Lambda};\Z)/Tor\cong \Lambda$ such that
$$h(A)=mg_{X_{\Lambda}}(A)$$
 whenever $h(A)\geq 0$.  
 $X_{\Lambda}$ can be chosen to be of the form $Y\# l (S^1\times S^3)$,  where $l=b_1(\Lambda)-\tilde b_1(\Lambda)$, and 
$Y$ is an appropriate symplectic manifold with Kodaira dimension $=-\infty$ or $0$ from the list below. 

 a) $\C\PP^2 $ connected sum with up to 9 $\overline{\C\PP^2}$, 

b) $S^2 \times S^2$,

c)  Enriques surface,

d)  $S^2$-bundle over  $T^2$,
 
\end{theorem}

As a consequence of Theorems \ref{main} and \ref{optimal},  when restricted to 
classes with non-negative square, $X_{\Lambda}$ has the smallest minimal genus function among 4-manifolds with the same cohomology algebra (modulo torsion).

In the last section, we apply our bound and give constraints on classes with non-negative square 
and which are represented by spheres. 
In a future paper we will make some explicit calculations of $h$ when   $2\tilde \chi(\Lambda)+3\sigma(\Lambda)<0$ and discuss further properties of the function $h$. In particular, we hope to  apply these calculations,
together with
Theorem \ref{main}, to completely  constrain classes represented by spheres.

 We end the introduction with a number of remarks.

$h(A)$ only depends on the cohomology algebra of $X$.  In this sense, it is similar to the Rokhlin (\cite{R}) and Hsiang-Szczarba (\cite{HS}) bounds via branched covering and Atiyah-Bott $G$-signature
 theorem. When $b^+(X)=1$, the bound given by $h$ is generally stronger. 

When $\Gamma$ is indefinite and $b^+\geq 2$,  there also exists adjunction type inequality which
depends only on the cohomology algebra for a few manifolds of small even  intersection forms  on the  $11/8$ line
 $l(-2E \oplus  3U)$, $l=1, 2, 3$ (\cite{MS},   \cite{FK}). In these cases, $c=0$ is the only adjunction class. 
 Beyond the $11/8$ line, no such adjunction inequality could exist, since Wall showed in \cite{Wall} that 
there exist manifolds such that every primitive ordinary class is represented by spheres.
This is also true for any odd and strongly indefinite intersection form: $m\langle 1\rangle \oplus n\langle -1\rangle$ with $m, n\geq 2$.




Similar to the $b^+>1$ Seiberg-Witten adjunction inequalities in  \cite{MST96}, \cite{OS}, the bounds $h_c$ involve the absolute value
$|c\cdot A|$.
The absolute value  cannot be removed by checking the case of $\C\PP^2$.



 \vs \noindent {\it Acknowledgment}. The research for the first named
author is partially supported by NSFC Grants 11371211 and 11431001. The research
for the second named author is partially supported by MST
postdoctoral fellowship. The research for the third named author
is partially supported by NSF. The authors would like to thank the referees for carefully reading the paper and for helpful comments and suggestions.

\section{Cohomology algebras of $b^+=1$ type and properties of $h$}
\subsection{Cohomology algebras and quadratic forms}\label{algebra}

\begin{definition}
Let  $\Lambda=\oplus _{i=0}^4\Lambda^i$ be a finitely generated, graded, commutative, associative algebra over $\Lambda_0=\Z$
with each summand $\Lambda^i$  a  free abelian group,  and a group isomorphism $p: \Lambda^4\cong  \Z$.
$(\Lambda, p)$, or simply $\Lambda$,  is called a cohomology algebra if,  
with respect to $p$, the products $\Lambda^i\times  \Lambda^{4-i}\to  \Lambda^4\cong \Z$ are duality pairings, in the sense that 
$$\Lambda^i\to  \Hom_{\Z} (\Lambda^{4-i}, \Z)$$
are isomorphisms of groups.

Denote the rank of $\Lambda^i$ by $b_i(\Lambda)$.
Let $\tilde b_1(\Lambda)$ be the rank of the skew-symmetric pairing $T:\Lambda^1\times \Lambda^1\to  \Lambda^2$.
Let $\chi(\Lambda)=\sum (-1)^i b_i(\Lambda)$  be the Euler number, and  $\tilde\chi(\Lambda)=2+b_2(\Lambda) -2\tilde b_1(\Lambda)$  the  modified Euler number.
Denote the signature type  of   $\Gamma:\Lambda^2 \times \Lambda^2\to  \Lambda^4 \cong \Z$ 
by $(b^+(\Lambda),  b^-(\Lambda))$, and let  $\sigma(\Lambda)=b^+(\Lambda)-b^-(\Lambda)$ denote the signature. 
$\Lambda$ is called a cohomology algebra of $b^+=k$ type if  $b^+(\Lambda)=k$. 

\end{definition}

\begin{example}  The following cohomology algebra of $b^+=0$ type
$\Lambda_S$ is modeled on $S^1\times S^3$:  $\Lambda_S^2=0$, $\Lambda_S^1=\Lambda_S^3=\Z$ with trivial $T$ pairing. 
\end{example}

Given two cohomology algebras $(\Lambda_a,p_a)$ and $(\Lambda_b,p_b)$, their direct sum $\Lambda_a\oplus \Lambda_b$ is  defined as the  following cohomology algebra $(\Lambda, p)$
 with
$\Lambda^{i}=\Lambda^i_a\oplus \Lambda^i_b$ for $1\leq i\leq 3$,  and $\Lambda^0 =\Z$.
To describe $\Lambda^4$ and $p:\Lambda^4\to \Z$, consider   the group homomorphism
 $$cs:   \Lambda^4_a\oplus \Lambda^4_b  \to  \Z\oplus \Z\to \Z$$
 sending $x=x_a +x_b\in \Lambda^4_a\oplus \Lambda^4_b$ to $p_a(x_a)+ p_b(x_b)$. 
Then   $\Lambda^4$ is the quotient of $\Lambda^4_a\oplus \Lambda^4_b$ by the kernel of  $cs$, and $p$ is the corresponding quotient homomorphism. 
   

With the product structure defined in the obvious way, it is not hard to verify that $\Lambda$ is a cohomology algebra.
Notice that the $b^+$  type is additive with respect to this operation.  

A simple but useful fact  is that   the sum  $\Lambda\oplus l\Lambda_S$ and $\Lambda$ have the same $b^+,\tilde b_1,\tilde \chi, \Lambda^2$ and $\Gamma$ for any non-negative integer $l$.

\subsubsection{Special features of $b^+=1$ type}  \label{b+=1algebra}
Suppose that $\Lambda$ is a cohomology algebra of $b^+=1$ type. We prove here the special  features  mentioned
in the introduction.

\begin{itemize}

\item The image of $T$ is either $0-$ or $1-$ dimensional.  We explain this conclusion by contradiction, following the proof of Lemma 2.4 in  \cite{LiLiu1}.  Suppose $\omega_1 =y_1 y_2 \in \text{Im}\, T$, and $\omega_2 =y_3 y_4 \in \text{Im}\, T$, where $y_1,\dots, y_4 \in \Lambda^1$, and $\omega_1,\ \omega_2$ are linearly independent in $\Lambda^2$. It is easy to see that $\omega_1^2 =0=\omega_2^2$. Since $b^+=1$, and 
$\omega_1,\ \omega_2$ are linearly independent, by the light cone lemma in \cite{LiLiu2}, $\omega_1 \omega_2 \not= 0 \in \Lambda^4$. This implies that $y_1,\dots, y_4 \in \Lambda^1$ are linearly independent and generate an exterior algebra.  Moreover, $y_iy_j$, $1\leq i<j\leq 4$, are in Im$T$, and are linearly independent in $\Lambda^2$. The restriction of the bilinear form $\Gamma$ to the linear span of $\{ y_iy_j| 1\leq i<j\leq 4 \}$  has vanishing signature. This implies $b^+(\Lambda) \geq 3$, a contradiction.  

\item Since $T$ is skew symmetric, $\tilde b_1 (\Lambda)$ is always an even number. 

\end{itemize}

\begin{definition} A cohomology algebra $\Lambda$ is called Lefschetz if $T$ is non-degenerate, i.e. $\tilde b_1(\Lambda)=b_1(\Lambda)$. A cohomology subalgebra  $\Lambda'$ is called a Lefschetz reduction of $\Lambda$ if
$\Lambda'$ is Lefschetz and $\Lambda\cong \Lambda'\oplus l\Lambda_S$ for some $l$.

\end{definition}

By Lemma \ref{Lef} below, 
Lefschetz reductions always exist, and are unique up to isomorphism when $\Lambda$ is of  $b^+=1$ type. We will denote any Lefschetz reduction of $\Lambda$ by $\Lambda_{red}$. 

\begin{lemma} \label{Lef}
Any cohomology algebra $\Lambda$ of $b^+=1$ type has Lefschetz reductions, and all Lefschetz reductions are isomorphic as cohomology algebras. 
\end{lemma}

\begin{proof}
If $T$ is trivial, then $\Lambda'=\Lambda^0 \oplus \Lambda^2 \oplus \Lambda^4$ is the unique Lefschetz reduction of $\Lambda$. Suppose $T$ is non-trivial. 
Let $\Lambda_0^1=\ker T=\{x\in \Lambda^1|xy=0, \forall y\in \Lambda^1\}$. 
It is clear that rank$(\Lambda_0^1)=b_1-\tilde{b}_1$.
Choose a complementary subspace $\tilde\Lambda^1$ such that $\Lambda^1=\tilde\Lambda^1\oplus\Lambda^1_0$. 
$\Lambda^3$ is also decomposed as $\Lambda^3=\tilde\Lambda^3\oplus\Lambda^3_0$ where $\tilde\Lambda^3=(\Lambda^1_0)^\perp$, $\Lambda^3_0=(\tilde\Lambda^1)^\perp$.
Note that $\tilde\Lambda^3$ is independent of choice of $\tilde\Lambda^1$ and $\tilde\Lambda^1\times \tilde\Lambda^3\rightarrow \Z$, $\Lambda^1_0\times \Lambda^3_0\rightarrow \Z$ are duality pairings.

Let $\Lambda' =\Lambda^0 \oplus \tilde \Lambda^1 \oplus \Lambda^2 \oplus \tilde\Lambda^3 \oplus \Lambda^4$ and $\Lambda'' =\Lambda^0 \oplus \Lambda^1_0 \oplus 0 \oplus \Lambda^3_0 \oplus \Lambda^4$. We claim that $\Lambda'$ is a Lefschetz reduction of $\Lambda$ and $\Lambda\cong\Lambda' \oplus \Lambda''\cong \Lambda' \oplus  (b_1-\tilde b_1)\Lambda_S$.
It is enough to verify that Im$(\Lambda^1_0\times\Lambda^2)=0$ and Im$(\tilde\Lambda^1\times\Lambda^2)\subset \tilde \Lambda^3$.

If Im$(\Lambda^1_0\times\Lambda^2)\neq 0$, there exist $x\in\Lambda^1_0, z\in\Lambda^2$ such that $xz\neq 0\in\Lambda^3$.
Then $xzy\neq 0$ for some $y\in \Lambda^1$.
This implies $xy\neq 0, x\notin \Lambda^1_0$, a contradiction.

For any  $x\in\tilde \Lambda^1, z\in\Lambda^2, y\in\Lambda^1_0$, we have $xzy=xyz=0$. So $xz\in (\Lambda^1_0)^\perp=\tilde \Lambda^3$. 
Hence Im$(\tilde\Lambda^1\times\Lambda^2)\subset \tilde \Lambda^3$. In fact,  Im$(\tilde\Lambda^1\times\Lambda^2)$ contains $\tilde b_1=\text{rank}(\tilde \Lambda^3)$ linearly independent vectors in $\tilde \Lambda^3$.  Suppose $\{ x_1,x_2,\dots, x_{\tilde b_1}\}$ is an integral basis for $\tilde \Lambda^1$, $F$ is a generator for Im$T$, and $B\in \Lambda^2$ such that $F\cdot B=1$. Such a $B$ exists since $\Gamma$ is unimodular. Then one can verify that $x_i B$, $i=1,2,\dots, \tilde b_1$, are linearly independent in $\tilde \Lambda^3$. 

Suppose $M'$ is another Lefschetz reduction of $\Lambda$, and $\Lambda \cong M' \oplus  M'' \cong M' \oplus  (b_1-\tilde b_1)\Lambda_S$. Suppose $M' =\Lambda^0 \oplus \tilde M^1 \oplus \Lambda^2 \oplus \tilde M^3 \oplus \Lambda^4$ and $ M'' =\Lambda^0 \oplus M^1_0 \oplus 0 \oplus  M^3_0 \oplus \Lambda^4$. It is easy to see that $ M^1_0 =\ker T=\Lambda^1_0$, and $\tilde M^1$ is another complement of $\ker T$ in $\Lambda^1$.  We also have $\tilde M^3 \supset \text{Im}(\tilde M^1\times\Lambda^2) =\text{Im}(\Lambda^1\times\Lambda^2)$.
The argument in the previous paragraph implies that $\tilde M^3 =(\ker T)^\perp =\tilde \Lambda^3$.  

Notice that $T$ reduces to a non-degenerate skew-symmetric bilinear form on $\Lambda^1/\ker T$, and $\tilde \Lambda^1$ and $\tilde M^1$ are both isomorphic to $\Lambda^1/\ker T$. Now we see that $\Lambda'\cong M'$ as cohomology algebras. 
 \end{proof}

\subsubsection{Unimodular quadratic forms}
Notice that  the symmetric bilinear form $\Gamma$
is unimodular. 
We will often abbreviate $\Gamma(x, y)$ as $x\cdot y$.
It induces a unimodular quadratic form
$Q:\Lambda^2\rightarrow \mathbb{Z}$ as $Q(x)=\Gamma(x, x)$. $Q(x)$
is called the  norm of $x$. $\Gamma$ is of even type if $Q(x)$ is
even for any vector $x \in \Lambda^2$. Otherwise, $\Gamma$ is called of
odd type. 
$\Gamma$
is called {\it definite\/}, or {\it
indefinite\/} if min$\{b^+,b^-\}=0$ or $\geq 1$ respectively.

The following classification is well known (eg. \cite{Wall-uni}, Theorem 5):
indefinite unimodular symmetric forms are classified  by their rank, signature and type.
Let $U=\begin{pmatrix} 0 & 1 \\ 1 & 0 \end{pmatrix}$ and $E$ be the hyperbolic lattice and the (positive definite) $E_8$
lattice respectively. The list of indefinite unimodular symmetric forms is
$$m\langle 1 \rangle \oplus n\langle -1\rangle, \quad pU\oplus qE, \
 \text{where} \ m,n,p\in\mathbb{N},q\in \mathbb{Z}.$$
Therefore,  when $b^+=1$, the list of unimodular symmetric forms is 
  $$\langle 1 \rangle \oplus n\langle -1\rangle,  \quad  U\oplus  q(-E),  \
 \text{where} \ n,  q\in \mathbb N \cup \{0\}.$$



\subsection{The adjunction classes, $Aut(\Gamma)^T$ and $h$}

Let $\Lambda$ be a cohomology algebra of $b^+=1$ type. 

Let $\mathcal C_{\Lambda}$ be the set of  adjunction classes of $\Lambda$. We discuss 
properties of this set and their consequences for the function $h$. 

Let $Aut(\Gamma)$ be the group of automorphisms of $\Lambda^2$ that preserve the bilinear form $\Gamma$. We shall concern with a subgroup of  $Aut(\Gamma)$ which preserves the set of adjunction classes. 

\begin{definition}Let $Aut(\Gamma)^T$ consist of elements $\phi\in Aut(\Gamma)$ such that $\phi(\text{Im}T)\subseteq  \text{Im}T$ (Actually, $\phi(\text{Im}T)=\text{Im}T$).  
\end{definition}

\begin{lemma}
$\mathcal C_{\Lambda}$ is preserved by $Aut(\Gamma)^T$.  Consequently, $h$ is invariant under $Aut(\Gamma)^T$.
\end{lemma}

\begin{proof}
We first show that $\mathcal C_{\Lambda}$ is preserved by $Aut(\Gamma)^T$. This is clear for type I adjunction classes, which are preserved by the larger group $Aut(\Gamma)$. For type II adjunction classes, the first condition is also preserved by $Aut(\Gamma)$ and the second condition is why we introduce $Aut(\Gamma)^T$.

We now establish the invariance of 
$h$    under   the action of $Aut(\Gamma)^T$.      Let $A\in \Lambda^2$ and $\alpha\in Aut(\Gamma)^T$. For any $c\in \mathcal C_{\Lambda}$, 
$h_c(A)=h_{\alpha(c)}(\alpha(A))$. Since $\alpha(c)$ is also in $\mathcal C_{\Lambda}$, we
have $h_c(A)\leq h(\alpha(A))$ for any $c\in \mathcal C_{\Lambda}$ and hence $h(A)\leq h(\alpha(A))$. Similarly, we also have $h(\alpha(A))\leq h(A)$.
Hence we have the desired equality. 
\end{proof}

By this invariance property, to determine the function $h$, it suffices to  pick an element, called reduced element,  in each orbit of $Aut(\Gamma)^T$ and calculate the value of $h$. We are able to carry this out  completely when $2\tilde \chi+3\sigma\geq 0$. 

The following simple fact will be useful in Section 4. 

\begin{lemma} \label{reduction}  Under the natural isomorphism between $\Lambda^2$ and $\Lambda_{red}^2$, $\mathcal C_{\Lambda_{red} } =\mathcal C_{\Lambda}$. Consequently, $h_{\Lambda}=h_{\Lambda_{red}}$.

\end{lemma}

\subsection{Cohomology algebras with $b^+=1$ and  $2\tilde \chi+3\sigma\geq 0$}\label{nonneg inv}

Observe that conditions $b^+ =1$ and $2\tilde \chi+3\sigma\geq 0$ imply that $1\geq \sigma \geq 4\tilde b_1 -8$, so either 
$\sigma\geq -8$,  $\tilde b_1=0$ or $\sigma=0, \tilde b_1=2$. 
We point out that there can be
adjunction classes of type II that are not of type I only when $\tilde{b}_1=2$.
 The list of $\Gamma$ in this case is 
 $$\langle 1 \rangle \oplus n\langle -1\rangle, 0\leq n\leq 9, \quad   U\oplus  q(-E), q=0,1$$ 
 The algebra $\Lambda$ is divided into 5 cases according to $T$ and $\Gamma$: 
 \begin{enumerate}
 \item $T$ trivial and  $\Gamma=U$,
 \item $T$ trivial and $\Gamma=U\oplus (-E)$,
 \item $\tilde b_1(\Lambda)=2$ and $\Gamma=U$,
 \item $T$ trivial and  $\Gamma=\langle 1 \rangle \oplus n\langle -1\rangle, 0\leq n\leq 9$,
 \item $\tilde b_1(\Lambda)=2$ and $\Gamma=\langle 1\rangle \oplus \langle -1\rangle$.
 \end{enumerate}

 \subsubsection{$Aut(\Gamma)^T$ and reduced classes}

We now  introduce the notion of reduced classes (under a choice of basis) in the five cases. 
We also describe $Aut(\Gamma)^T$ explicitly in {\bf Cases  (1), (3), (4), (5)}.  
For this purpose, notice that $-Id\in Aut(\Gamma)^T$, and $Aut(\Gamma)^T=Aut(\Gamma)$
 whenever $T$ is trivial. 
Notice also that $\Gamma|_{\text{Im}T}=0$. 
 
 {\bf Case (1)} 
 Take a  basis $\mathcal B=\{F, B\}$ of $ \Lambda^2$ with $F\cdot F=B\cdot B=0$, $F\cdot B=1$. 
 The automorphism group  $Aut(\Gamma)^T=\Z_2\oplus \Z_2$ generated by $-Id$ and switching of $F$ and $B$. 
 
 Reduced classes:  $aF+bB$,  $a\geq |b| \geq 0$. 
 
{\bf Case (2)}
Consider a decomposition $ \Lambda^2= \Lambda^2_U \oplus  \Lambda^2_{-E}$ such that $\Gamma|_{\Lambda^2_U}\cong U,\Gamma|_{\Lambda^2_{-E}}\cong -E$. Take a basis $\mathcal B=\{F, B\}$ of $ \Lambda^2_U$ with $F\cdot F=B\cdot B=0$, $F\cdot B=1$. 



Reduced classes: $aF$ with $a\geq 0$, or $aF+bB+\xi$,  $\xi\in \Lambda^2_{-E}$ with $\displaystyle a\geq |b|> \frac{\sqrt{-\xi\cdot\xi}}{2}\geq 0$. 

{\bf Case (3)} 
Assume ${\rm Im}\, T$ is generated by $F$. In particular, $F\cdot F=0$.  Choose
a basis $\mathcal B=\{F, B \}$ of $ \Lambda^2$
 with $B\cdot B=0$ and $F\cdot B=1$.  
 
The automorphism group  $Aut(\Gamma)^T \cong\Z_2$ is generated by $-Id$ since   $F$ is preserved up to sign.  
 
Reduced classes:  $a F + bB$ with  $a>0$, or $a=0, b\geq 0$.

 {\bf Case (4)} 
 Let $n=b^-\leq 9$ and $\mathcal{B}=\{H, E_1,\dots, E_n\}$ be an orthogonal basis for $\Lambda^2$ such that $H\cdot H=1=-E_i\cdot E_i$.  

According to Wall \cite{Wall-uni-ii} 1.6.,  $Aut (\Gamma)^T$ is generated by reflections along
the classes $H, E_i, H-E_1-E_2-E_3,  E_i-E_j, i\ne j$.


Reduced classes:   $aH-\sum_{i=1}^n b_iE_i$ with $b_1\geq b_2\geq\cdots\geq b_n \geq 0$ and 
$$a\geq\left\{\begin{array}{ll}
b_1 &\hbox{if }n=1\\
b_1+b_2 &\hbox{if }n=2\\
b_1+b_2+b_3 &\hbox{if }n\geq 3\\
\end{array}\right.$$

{\bf Case (5)} 
In this case $\sigma=0$ and the intersection form is equivalent to  $V=\begin{pmatrix} 0 & 1 \\ 1 & 1 \end{pmatrix}$.    
Suppose ${\rm Im}\, T$ is generated by $F$. Notice that $F\cdot F=0$. We choose 
a basis $\mathcal B=\{F, B \}$ of $ \Lambda^2$
 with   $B\cdot B=F\cdot B=1$. 

The only non-trivial element of $Aut(\Gamma)^T$ is  $-Id$, since $F$ is preserved up to sign. 

Reduced classes: $a F + bB$ with  $a>0$, or $a=0, b\geq 0$.

\begin{lemma}\label{unique}
In each of the  five cases,  for any class $A\in \Lambda^2$ with $A\cdot A\geq 0$,  there exists a reduced class lying in the $Aut(\Gamma)^T$-orbit of $A$. 
Actually, such class is unique except in {\bf Case (2)}.
\end{lemma}

\begin{proof}
This is straightforward to see using the description of $Aut(\Gamma)^T$ above for {\bf Cases  (1), (3),  (5)}.  
For {\bf Case (4)} this was established by \cite{Li}, Proposition~1  (This is still true for arbitrary $n$ by \cite {Li} and \cite{GZ}).

It remains to deal with  {\bf Case (2)}. Let $A=aF+bB+\xi$ with $A\cdot A=2ab+\xi\cdot \xi\geq 0$. We may assume $A$ is primitive and $a\geq b>0$ because $\xi\cdot \xi\leq 0$. To find a reduced class in the $Aut(\Gamma)^T$-orbit of $A$, we will make induction on $m(A):=$min$(a,b)$.

Let us recall the following fact.
\begin{lemma}[\cite{Wall-uni-ii}, 5.10] \label{E8}
For all $\xi \in \Lambda^2_{-E}$ and nonzero integers $b$, we can find $\omega\in\Lambda^2_{-E}$ such that either $\xi+b\omega=0$, or 
$$0<|(\xi+b\omega)\cdot (\xi+b\omega)|<2b^2.$$
\end{lemma}

We also need the automorphism $E_\omega\in Aut(\Gamma)$ given by
\begin{align*}
E_\omega (F)&=F,\\
 E_\omega (B)&=B+\omega-N(\omega)F,\\
  E_\omega (\xi)&= \xi-(\xi\cdot\omega)F, \forall \xi \in \Lambda^2_{-E}
 \end{align*}
where $N(\omega)=\frac{\omega\cdot \omega}{2}$. 

 If $A$ is not reduced, apply Lemma \ref{E8} to $\xi$ and $b$, there exists $\omega\in \Lambda^2_{-E}$ such that $\xi+b\omega=0$ or 
$$0<|(\xi+b\omega)\cdot (\xi+b\omega)|<2b^2$$
Let $A'=E_\omega(A)=(\xi+b\omega)+(a-bN(\omega)-\xi\cdot\omega)F+bB$. $A'$ can be non-reduced only when $a-bN(\omega)-\xi\cdot\omega<b$ and $|(\xi+b\omega)\cdot (\xi+b\omega)|\geq 2(a-bN(\omega)-\xi\cdot\omega)^2$. In this case, $m(A')=a-bN(\omega)-\xi\cdot\omega<m(A)$. We may now interchange the role of $F, B$ and repeat the process by induction.
\end{proof}

\begin{remark} The proof for {\bf Case (2)} can be generalized to classes with negative square and to $\Gamma=U\oplus kE$. It partially demonstrates a fact mentioned in \cite{Wall-uni}:
If $\Gamma$ is nearly definite (min$\{b^+,b^-\}=1$) and of even type, there are only finitely many equivalence classes of vectors of given norm, divisor and type.
\end{remark}

\subsubsection{Calculation of $h$ for reduced classes}

Let us first make a few   observations. 

1. Given $A$, we want to minimize  $|c\cdot A|$ among all adjunction classes.


2. When $T$ is trivial there are only type I adjunction classes. 

3. Notice that $2\tilde{\chi}+3\sigma=\sigma+8-4\tilde{b}_1$. So  we need to consider both type I and   type II adjunction classes   when $\tilde{b}_1\geq 2$.

We  assume $A$ is a nonzero class of non-negative square. 

\begin{prop}\label{h for reduced}
Suppose $2\tilde \chi(\Lambda)+3\sigma(\Lambda)\geq 0$ and $\Lambda$ is as in {\bf Cases  (1), (2), (3),  (5)}. 
Suppose a basis $\mathcal B$ is given as in Section 2.3.1 and $A\in \Lambda^2$ is a reduced 
class with $A\cdot A\geq 0$.  There is an adjunction class $c_0$, independent of $A$,  
such that
$h(A)=h_{c_0}(A)$. Explicitly, with respect to the basis $\mathcal B$, $c_0$ is given in each case as follows. 

\begin{itemize}

\item {\bf Case (1)} $c_0=2F+ 2B$,

\item {\bf Case (2)} $c_0=0$,

\item {\bf Case (3)} $c_0=2B$,


\item {\bf Case (5)} $c_0=F-2B$.

\end{itemize}

\end{prop}

\begin{proof}
 
{\bf Case (1)}.
Suppose $A =aF+bB$ is reduced, namely,  $a\geq |b|\geq 0$. 
Then $A\cdot A=2ab\geq 0$ implies $b\geq 0$. 

Any  characteristic class is of the form  $c =2kF+2lB$.  
We also assume that  $k\geq 0$ since we try to minimize the absolute value $|c\cdot A|$. 

Since $T$ is trivial, there are only  type I adjunction classes. Such a class
$c$ should satisfy
$$c\cdot c=8kl>\sigma=0,  $$
so $k, l>0$. 
Since $a, b\geq 0$ and $k, l>0$, 
$$|c\cdot A|=|2al+2bk|=2al+2bk,$$ and it  reaches the minimum when $k=l=1$.


For $c_0=2F+2B$ and $A$ a reduced class,  we obtain
\begin{equation} h(A)=h_{c_0}(A)= (a-1)(b-1). \label{ha-1} \end{equation}

{\bf Case (2)}.
Again there are only  type I adjunction classes since $T$ is trivial.  The class $0$  is characteristic and $\sigma=-8<0$, so we  take $c_0=0$. 
Clearly,  $|c_0\cdot A|=0$ reaches its minimum for any $A$, and thus  
\begin{equation} \label{enriques} h(A)=h_{c_0}(A)=\frac{1}{2}  (A\cdot A +2).
\end{equation}

Notice that this formula is in fact valid for an arbitrary class with non-negative square, not just
the reduced ones.

{\bf Case (3)}.
Let $A=aF +bB$ be a reduced class, namely, $a> 0$, or $a=0, b\geq 0$. 
Since  $A\cdot A=2ab\geq 0$, 
we always  have $b\geq 0$.

Any characteristic class is of the form $c=2qF+2pB$ with $q\geq 0$ up to the automorphism $-Id$.

We first consider type II adjunction classes. Such a class $c$ satisfies 
 $$c\cdot c=8pq\geq 2\tilde \chi+3\sigma=0\quad {  and }  \quad c\cdot F=2p\ne 0.$$
Notice that these conditions imply that $p >0$.
Since $a, b, p, q\geq 0$ and $p> 0$,  $|c\cdot A|=|2ap+2bq|$ is minimized by $p=1$ and $q=0$.
Therefore, we take  $c_0=2B$ and have 
\begin{equation}
h(A)= h_{c_0}(A)= a(b-1) +1.  \label{ha-3}
\end{equation}

As for type I adjunction classes, observe that $\sigma=0$. So any such class also has nonzero $p, q$ coefficients and is of type II. Therefore $h(A) = h_{c_0}(A)$.

{\bf Case (5)}. This  case is similar to {\bf Case (3)}. 
Suppose $A=aF+bB$ is a reduced class, either $a>0$ or $a=0, b\geq 0$. 
Then $A\cdot A=(2a+b)b\geq 0$ implies either $b\geq 0$ or $2a+b\leq  0$.

Any  characteristic class is of the form $c=(2q-1)F+2pB$. And we will assume that $2q-1\geq 0$.  We first consider type II adjunction classes.
Such a class satisfies 
$$c\cdot c=4p(p+2q-1)\geq 2\tilde \chi+3\sigma=0\quad {  and }  \quad c\cdot F=2p\ne 0.$$
Thus either $p>0$ or $p+2q-1\leq 0$.  
To minimize $|c\cdot A|$, we write
$$|c\cdot A|=|(2a+b)p+(p+2q-1)b|.$$
In each case, we can show that  $|c\cdot A|$  is minimized by taking 
$p+2q-1=0$ and $p=-1$. 

Therefore, we take $c_0=F-2B$ and have 
\begin{equation} h(A)=h_{c_0}(A)=\frac12 (b-1)(2a+b) +1.
\label{ha-5} \end{equation}
As for type I adjunction classes, observe that $\sigma=0$. So any such class also has nonzero $p, q$ coefficient and is of type II. Therefore $h(A) = h_{c_0}(A)$.

\end{proof}


For {\bf Case(4)} we have the following analogous statement. Since the proof is long and elementary,
we defer it to the Appendix. 

\begin{lemma} \label{case4} Suppose $\Lambda$ is as in {\bf Case (4)} and $A=aH-\sum_{i=1}^n b_i E_i$ is a reduced class with
$A\cdot A\geq 0$. Then we take
$c_0=3H-\sum_{i=1}^{n} E_i$ and have
\begin{equation}
h(A)= h_{c_0}(A)= \frac{(a-1)(a-2)-\sum b_i(b_i-1)}{2}  =\binom{a-1}{2} -\sum_{i=1}^n \binom{b_i}{2}.\label{negoddgenus}
\end{equation}

\end{lemma}







We remark that it is easy to check that  
$c_0$ in all five cases is characterized by  the following conditions: 

\begin{itemize}
\item $c_0$ is  characteristic,    

\item $c_0\cdot c_0=2\tilde \chi+3\sigma$,  

\item $c_0$  itself  is reduced, 

\item $c_0$ pairs non-trivially with Im$T$ when $T$ is non-trivial.

\item $c_0$ has smallest coefficients in absolute value among classes satisfying the above four conditions. 

\end{itemize}

\subsubsection{The sign of $h$}

We explore Proposition \ref{h for reduced} to determine when $h$ takes positive and non-negative values.

\begin{cor} \label{leq0} Suppose $2\tilde \chi(\Lambda)+3\sigma(\Lambda)\geq 0$ and $\Lambda$ is 
in {\bf Case (1), (3), (5)}. 
If $A=aF+bB\in \Lambda^2$ with $A\cdot A\geq 0$
 is reduced and  $h(A)\leq 0$, then $A$ is   given as follows.
\begin{itemize}

\item  {\bf Case (1)} $h(A)< 0$ when $a>1, b=0$, and $h(A)=0$ when $a\geq b=1$ or $(a, b)=(1,0), (0,0)$.  


\item  {\bf Case (3)} $h(A)< 0$ when $a>1, b=0$, and $h(A)=0$ when $(a, b)=(1,0), (0,0)$. 


\item {\bf Case (5)} $h(A)< 0$ when  $a>1, b=0$, and $h(A)=0$ when $(a, b)=(1,0), (0,0)$. 
\end{itemize}

\end{cor}

\begin{proof}
The formulae  for $h(A)$ are given explicitly in \eqref{ha-1}-\eqref{ha-5}  for the three cases. We can solve $h(A)\leq 0$ and $h(A)=0$ easily.

\end{proof}
For {\bf Case (2)}, any class $A\neq 0$ with $A\cdot A\geq 0$ has $h(A)> 0$. 

For {\bf Case (4)},  we can also use  \eqref{negoddgenus}  to determine when 
$h$ takes
positive and non-negative values. Again we will defer the proof to the Appendix since
it is elementary but complicated. 

\begin{lemma} \label{leq0case4}
For {\bf Case (4)}. 
$h_{}(A)=0$ when    $A=H-E_1, H, 2H$ or $$a=b_1+1, b_1\geq 1, \, \, b_2=0, 1, \, \,b_i=0, i\geq 3,$$
$h_{}(A)<0$ when $A=a(H-E_1)$ with $a\geq 2$. 
\end{lemma}

\section{Proof of Theorem \ref{main}}

\subsection{Seiberg-Witten invariants}
Let $X$ be a smooth, closed, connected, oriented four-manifold. Consider the exterior algebra  $$V(X)=\wedge^* H^1(X;\Z)/Tor.$$
In this section, we review  the $V(X)$-valued  Seiberg-Witten invariants.
For more details, see \cite{Mor,OT,LiLiu3}.

Suppose  $g$ is a Riemannian metric on $X$.
Let $P\rightarrow X$ be the $SO(4)$-frame bundle associated to the tangent bundle.
A spin$^c$ structure on $X$ is a lifting $\tilde{P}$ of $P$ to a $Spin^c(4)$-principal bundle.
The associated complex spinor bundle $S_\C(\tilde{P})\rightarrow X$ is decomposed as the direct sum of two rank-2 Hermitian vector bundles $S^\pm_\C(\tilde{P})$.  
The determinant line bundle of $\tilde{P}$ is defined by
$\mathcal{L}=\det (S^\pm_\C(\tilde{P}))$. It is known that $c_1(\mathcal{L})\in H^2(X;\Z)$ is characteristic and spin$^c$ structures are parametrized by characteristic classes if $H^2(X;\Z)$ has no 2-torsion.
For convenience, we also use $\mathcal{L}$ to denote $\tilde{P}$.
A unitary connection $A\in \mathcal{A}_\mathcal{L}$ on $\mathcal{L}$ and the Levi-Civita connection on $TX$ induce a Dirac operator
$$D_A: \Gamma(S^+_\C(\tilde{P}))\rightarrow \Gamma (S^-_\C(\tilde{P})).$$

For  any real-valued self-dual 2-form $\delta$ on $X$,  the (perturbed) Seiberg-Witten equations for a connection $A\in \mathcal{A}_\mathcal{L}$ and a section $\psi\in \Gamma (S^+_\C(\tilde{P}))$  are
$$ \begin{cases} D_A \psi=0, & \\
F_A^+ +i \delta=\psi\otimes \psi^*-\frac{|\psi|^2}{2}Id . \end{cases} $$

The gauge group $\mathcal{G} =\text{Map}(X, S^1)$ acts on the configuration space $\mathcal{C}(\tilde{P})= \mathcal{A}_\mathcal{L} \times \Gamma (S^+_\C(\tilde{P}))$. A pair $(A,\psi)$ is called reducible if $\psi\equiv 0$ and irreducible otherwise. Let $\mathcal{C}^*(\tilde{P})$ be the subset of irreducible pairs, on which gauge group acts freely. Let $\mathcal{B}(\tilde{P})=\mathcal{C}(\tilde{P}) /\mathcal{G}$ and $\mathcal{B}^*(\tilde{P}) =\mathcal{C}^*(\tilde{P}) /\mathcal{G}$ be the quotient spaces. 

Let $\mathcal{M}_{X,g,\delta}(\mathcal{L})$ be the quotient space of elements in $\mathcal{C}(\tilde{P})$ satisfying the Seiberg-Witten equations under the action of gauge group.
$\mathcal{M}_{X,g,\delta}(\mathcal{L})$ is always compact and, for generic $(g,\delta)$, it is a smooth oriented manifold of dimension 
$$k(\mathcal{L})=\frac{c_1(\mathcal{L})^2-(2\chi(X)+3\sigma(X))}{4}$$
in $\mathcal{B}^*(\tilde{P})$ (cf. \cite{KM} Lemma 5, \cite{W} Theorem 6.1.1). The orientation of $\mathcal{M}_{X,g,\delta}(\mathcal{L})$ is determined by an orientation of $H^1(X;\R) \oplus H^2_+ (X;\R)$. 

\subsubsection{$V(X)$-valued Seiberg-Witten invariant}
There is a natural line bundle over $X\times \mathcal{B}^*(\tilde{P})$, defined by $(X\times \mathcal{C}^*(\tilde{P}) \times \C )/\mathcal{G}$, where the action is given by $f \cdot (x, (A,\psi), z) =(x, \, f\cdot (A,\psi), f^{-1}(x)\, z)$, for any $f\in \mathcal{G}$, $x\in X$, $(A,\psi)\in \mathcal{C}^*(\tilde{P})$ and $z\in\C$. This action is free, and the quotient space is a line bundle over $X\times \mathcal{B}^*(\tilde{P})$, called the universal line bundle. 
 Its first Chern class $\mu \in H^2(X\times \mathcal{B}^*(\tilde{P});\Z)$ induces a linear map by slant product:
$$\phi: H_i(X;\Z) \to H^{2-i}(\mathcal{B}^*(\tilde{P}); \Z), \quad 0\leq i\leq 2, $$
where $\phi(a) =\mu/a$, for any $a\in H_i(X;\Z)$.

For  generic $(g,\delta)$, we  introduce the  $V(X)$-valued invariant ${\bf SW}(X,g,\delta,\mathcal{L})$. For a collection of homology classes $\gamma_1,\gamma_2,\dots, \gamma_p \in H_1(X;\Z)/Tor$, the Seiberg-Witten function is defined as (see \cite{LiLiu3}, \cite{OT})
$${\bf SW}(X,g,\delta,\mathcal{L})(\gamma_1\wedge\cdots \wedge \gamma_p)=<  \phi(\gamma_1) \cup\cdots \cup\phi(\gamma_p) \cup \phi(x_0)^{(k(\mathcal{L})-p)/2},[\mathcal{M}_{X,g,\delta}(\mathcal{L})]>$$
where $x_0$ is any point on $X$.
The above invariant is trivial unless $p \equiv k(\mathcal{L})\ (\text{mod}\ 2)$.  Note that
${\bf SW}(X,g,\delta, \mathcal{L})(1)$, denoted by $SW(X,g,\delta,\mathcal{L})$, is the usual Seiberg-Witten invariant. 

When $b^+>1$,  {\bf SW}  is independent of the choice of  generic $(g,\delta)$ and is denoted by
${\bf SW}(X,\mathcal{L})$. 

When  $b^+=1$, {\bf SW}  depends on the choice of a chamber.
 For a $4$-manifold with $b^+=1$, the set of nonzero real second cohomology classes of non-negative square has two components. The orientation of $H^2_+ (X;\R)$ picks one component $\mathcal{C}$ of them and is called the forward cone.
In this case, each metric $g$ induces a unique self-dual harmonic 2-form $\omega_g\in \mathcal{C}$ with $[\omega_g]\cdot [\omega_g]=1$.
 The Seiberg-Witten invariant only depends on the sign of $(2\pi c_1(\mathcal{L})-\delta)\cdot [\omega_g]$ rather than the choice of $(g,\delta)$.

\begin{definition}  Let $X$ be a smooth, closed, connected, oriented 4-manifold with $b^+=1$ and a choice of cohomology orientation.
For any  $Spin^c$ structure $\mathcal L$ with $k(\mathcal L)\geq 0$, we define
$${\bf SW}_+(X,\mathcal{L})={\bf SW}(X,g,\delta,\mathcal{L})$$ 
where $(g,\delta)$ is generic with $(2\pi c_1(\mathcal{L})-\delta)\cdot [\omega_g]>0$.
$  {\bf SW}_-(X,\mathcal{L})$ is defined similarly. 
\end{definition}

Notice that  ${\bf SW}_{\pm}(X,\mathcal{L})$ depend on the choice of the forward cone. They 
are switched if we change the forward cone. 

We will study the minimal genus problem for a cohomology class $A$ with non-negative square.
Since $A$ and $-A$ have the same minimal genus,  we can assume that $A$
 lies in  the closure of the forward cone $\mathcal C$.

We recall the symmetry and the blow up formulas for ${\bf SW}_{\pm}$. 

\begin{lemma}  \label{formulae}

Let $X$ be a smooth, closed, connected, oriented 4-manifold with $b^+=1$, and $\mathcal{L}$  a spin$^c$ structure on $X$ with formal dimension $k(\mathcal{L})\geq 0$. 
Let  $E$ be the exceptional class in $H^2(X\sharp \overline{\C\PP}^2;\Z)$. The forward cone $\tilde{\mathcal{C}}$ of $X\sharp \overline{\C\PP}^2$ is chosen such that $\mathcal{C}\subset \tilde{\mathcal{C}}$. Then we have

\begin{itemize}

\item Complex conjugation  (\cite{W},\cite{Mor}) 
\begin{equation}\label {symmetry} 
{\bf SW}_+(X,\mathcal{L})( \gamma_1 \wedge\cdots \wedge \gamma_{p})=(-1)^{j(p)}  {\bf SW}_-(X, -\mathcal{L})(\gamma_1 \wedge\cdots \wedge \gamma_{p})
\end{equation}
where $j(p)=1 +\frac{1}{2}(p-b_1)$.

\item  Blow-up formula (\cite{FS}, \cite{LiLiu3})
\begin{equation}\label{special blowup} 
 {\bf SW}_{\pm}(X\sharp\overline{\C\PP}^2,\mathcal{L}\pm E)={\bf SW}_{\pm}(X,\mathcal{L}).
\end{equation}

\end{itemize}
\end{lemma} 

We explain \eqref{symmetry} briefly since it is only slightly different from Corollary 6.8.4 in \cite{Mor}. Looking at the definition of ${\bf SW}$ above and the proof of Corollary 6.8.4 in \cite{Mor}, 
we infer that the total effect of complex conjugation on the invariant is $(-1)^{\delta +\frac{k(\mathcal{L}) +p}{2}}$, where $\delta=1+b_1 +b^+ +\text{ind}_\C (D_A)$  counts the change of orientation of moduli space, and $\frac 12 (k(\mathcal{L})+p) =\frac12 \{ (b_1 -1-b^+) +2 \text{ind}_\C (D_A) +p \}$ counts the number of changes of $\phi \to -\phi$, since complex conjugation changes the first Chern class of the universal line bundle $\mu\to -\mu$. 
So the total effect is $(-1)^{j(p)}$, where we have used $b^+ =1$.

\subsubsection{Two types of invariants}
For type I adjunction classes, we need the following version of invariants,
$${\bf SW}_{\pm} (X,\mathcal{L}) ( \gamma_1 \wedge\cdots \wedge \gamma_{b_1})$$
where  
$$\{\gamma_1,\gamma_2,\dots, \gamma_{b_1}\}$$ is an integral basis for $H_1(X;\Z)/Tor$ compatible with the orientation of $H^1(X;\R)$.
This invariant is defined if $k(\mathcal L)\geq  b_1(X)$.  Let $\mathcal{L}$ be a spin$^c$ structure with $c_1 (\mathcal{L})=c$. Observe that 
$$k(\mathcal{L})-b_1=\frac{c^2-\sigma(X)}{4}-2$$ where we have used $b^+=1$. 
$k(\mathcal{L})\geq b_1$ is equivalent to $c^2-\sigma(X)\geq 8$, or $c^2>\sigma(X)$ since $c^2\equiv \sigma(X) \ (mod\ 8)$.

For type II adjunction classes, we need the following version of invariants:
$${\bf SW}_{\pm} (X, \mathcal {L})( \gamma_{\tilde b_1 +1} \wedge\cdots \wedge \gamma_{b_1})$$
where the set of homology classes
$$\{\gamma_{\tilde b_1 +1},\dots, \gamma_{b_1}\} \subset H_1(X; \Z)/Tor$$ is chosen as follows. Take an integral basis $\{ x_1,x_2,\cdots, x_{b_1} \}$ for $H^1(X;\Z)/Tor$ as in the proof of Lemma~\ref{Lef} such that $T$ is non-degenerate on the linear span of $\{ x_1,x_2,\cdots, x_{\tilde b_1} \}$ and $\{ x_{\tilde b_1 +1},\cdots, x_{b_1} \}$ is a basis for $\ker T$. 
We further assume that this basis is compatible with the orientation of $H^1(X;\R)$. 

Let $\{ \gamma_1,\gamma_2,\dots, \gamma_{b_1} \}$ be the dual basis for $H_1(X; \Z)/Tor$, i.e. $\langle x_i,\gamma_j\rangle=\delta_{ij}$. 
This version of invariant is defined if $k(\mathcal L)\geq b_1(X)- \tilde b_1(X)$.
Similarly observe that   $$k(\mathcal{L})-b_1 +\tilde b_1=\frac{1}{4}(c^2-2\tilde \chi-3\sigma).$$

When $b_1=\tilde b_1$, the invariants are  simply $SW_{\pm} (X, \mathcal {L})$.

The wall crossing numbers for these two types of invariants are given below.  
\begin{equation}  \label{special wallcross0} 
{\bf SW}_+(X,\mathcal{L})( \gamma_1 \wedge\cdots \wedge \gamma_{b_1})-{\bf SW}_-(X,\mathcal{L})(\gamma_1 \wedge\cdots \wedge \gamma_{b_1})= 1.
\end{equation}

\begin{align}  \label{special wallcross1} 
& {\bf SW}_+(X,\mathcal{L})( \gamma_{\tilde b_1 +1} \wedge\cdots \wedge \gamma_{b_1})-{\bf SW}_-(X,\mathcal{L})(\gamma_{\tilde b_1 +1} \wedge\cdots \wedge \gamma_{b_1})   \\
 &=k_1 \langle      \gamma_{\tilde b_1 +1} \wedge\cdots \wedge \gamma_{b_1} \wedge u_c^{\frac{\tilde b_1}{2}}, x_1 \wedge x_2 \wedge \cdots \wedge x_{b_1} \rangle \notag \\
 &=k_2 (c_1 (\mathcal{L})\cdot F)^{\frac{\tilde b_1}{2}} \notag
\end{align}
where $k_1=\pm \frac{1}{(\tilde b_1 /2)!}$, $k_2$ is a nonzero constant, $u_c =\frac12 \sum_{1\leq i <j\leq b_1} c_{ij} \gamma_i \wedge \gamma_j$, $c_{ij}=\langle c_1(\mathcal{L})\cup x_i \cup x_j, [X] \rangle$, and $F$ is a generator for Im$T$.   Formula \eqref{special wallcross0} follows from 
 \cite{KM,LiLiu1},
and  the first equality in \eqref{special wallcross1}  follows from Theorem 16 in \cite{OT}.  
The second equality in \eqref{special wallcross1}   follows from  direct computation. Assume $x_i x_j=d_{ij}F$. Then $c_{ij}=d_{ij} c_1 (\mathcal{L})\cdot F$. We have $k_2 \not=0$ since  $T$ is non-degederate on the linear span of $\{ x_1,x_2,\cdots, x_{\tilde b_1} \}$.

In light of the SW dimension calculation above
and wall crossing formulas \eqref{special wallcross0}  and \eqref{special wallcross1},      Theorem \ref{main}  follows from the following more general  result which involves the wall crossing value of the full ${\bf SW}$. 

\begin{theorem}\label{refined}
Let $X$ be a smooth, closed, connected, oriented 4-manifold with $b^+(X)=1$, and $\Sigma \subset X$ a connected, smooth, embedded surface with $\Sigma\cdot\Sigma \geq 0$. Assume $\mathcal{L}$ is a spin$^c$ structure such that its SW dimension $k(\mathcal L)\geq 0$, and the wall crossing value $${\bf SW}_+(X,\mathcal{L})-
{\bf SW}_-(X,\mathcal{L})$$ is nonzero (in $\wedge^* H^1(X;\Z)/Tor$).
Then
$$2g(\Sigma)-2\geq \Sigma\cdot\Sigma-|\langle c_1(\mathcal{L}), [\Sigma] \rangle|.$$
\end{theorem}





For convenience, we abbreviate $c_1(\mathcal{L})$ to $c$. 

\subsection{The case $\Sigma\cdot \Sigma\leq   |\langle c, [\Sigma] \rangle|$} \label{3.2}
Notice that $\Sigma\cdot \Sigma-|\langle c, [\Sigma] \rangle|$ is an even number since $c$ is a characteristic class. 
So if   $\Sigma\cdot \Sigma-|\langle c, [\Sigma]\rangle| < 0$,  $\Sigma\cdot \Sigma-|\langle c\cdot [\Sigma] \rangle|\leq -2$ and the adjunction inequality holds trivially. 

In the case $\Sigma\cdot \Sigma=|\langle c, [\Sigma]\rangle|$, we use the following result of Morgan-Szab\'o-Taubes (it is stated for $SW$ but certainly works for the full ${\bf SW}$): 

\begin{lemma}[\cite{MST96}, Lemma 10.2] \label{MST}
Let $X$ be a smooth, closed, connected, oriented 4-manifold with $b^+=1$, and $\Sigma \subset X$ a smooth embedded sphere with $\Sigma\cdot\Sigma=0$ and $[\Sigma]$ of infinite order.
Let $\mathcal L$ be a spin$^c$ structure and use the forward cone $\mathcal C$ of $[\Sigma]$ to define ${\bf SW}_{\pm}(X,\mathcal{L})$. Then
$${\bf SW}_-(X, \mathcal L)=0$$
whenever  $\langle c_1(\mathcal L), [\Sigma]\rangle=0$. 
\end{lemma}

Suppose that $\Sigma\cdot \Sigma >0$ and $\langle c, [\Sigma]\rangle =\Sigma\cdot \Sigma$. Blowup $X$ at $n=\Sigma\cdot \Sigma$ points. Let $\tilde X$ be the resulting manifold, $\tilde{\Sigma}$ the resulting surface, which has the same genus as $\Sigma$, and $E_1,\dots,E_n$ the exceptional classes in $H^2(\tilde X;\Z)$. Then $[\tilde \Sigma]=[\Sigma]-\sum \text{PD}(E_i)$. Consider the class
$$\tilde c=c-\sum E_i.$$
$\tilde c$ is still characteristic and has the same SW dimension as $c$. Moreover,  $\tilde \Sigma\cdot \tilde \Sigma=\langle \tilde c, [\tilde \Sigma] \rangle =0$. 
Thus we are reduced to the case that  $$\Sigma\cdot \Sigma=0, \quad \langle c, [\Sigma]\rangle =0.$$

In this case, we will show that $\Sigma$ is not a sphere by applying Lemma~\ref{MST} to both $c$ and $-c$. 

If $\Sigma$ is a sphere and $[\Sigma]$ is of infinite order, by  Lemma \ref{MST} applied to $\pm c$, ${\bf SW}_-(X, \pm c)=0$.

By the assumption on the non-trivial wall crossing  for $c$, there exist $\gamma_1, \cdots, \gamma_p\in H_1(X;\Z)$ such that
 ${\bf SW}_+(X, c)(\gamma_1 \wedge \cdots \wedge \gamma_{p})\ne 0$. 

By formula \eqref{symmetry} in Lemma~\ref{formulae}, 
$${\bf SW}_-(X,  -c)(\gamma_1 \wedge \cdots \wedge \gamma_{p})=\pm {\bf SW}_+(X, c)(\gamma_1 \wedge \cdots \wedge \gamma_{p}) \ne 0.$$  This contradicts  ${\bf SW}_-(X, -c)=0$.

\subsection{$\Sigma\cdot \Sigma>   |\langle c, [\Sigma]\rangle|$}

Since $c$ is characteristic, we actually have  $\Sigma\cdot \Sigma - |\langle c, [\Sigma]\rangle| \geq 2$ in this case.

\begin{prop}\label{bound}
Let $X$ be a smooth, closed, connected, oriented 4-manifold with $b^+=1$, and $\Sigma \subset X$ a connected smooth embedded surface with 
$$\Sigma\cdot\Sigma >0 \hbox{ and }   g(\Sigma)>0.$$ 
 Suppose  $\mathcal{L}$ is a spin$^c$ structure such that its SW dimension $k(\mathcal L)\geq 0$, the wall crossing value ${\bf SW}_+(X,\mathcal{L})-
{\bf SW}_-(X,\mathcal{L})$ is nonzero (in $\wedge^* H^1(X;\Z)/Tor$),  and 
$$|\langle c_1(\mathcal{L}), [\Sigma]\rangle |< \Sigma\cdot\Sigma.$$
Then
$$2g(\Sigma)-2\geq \Sigma\cdot\Sigma-|\langle c_1(\mathcal{L}), [\Sigma]\rangle|.$$
\end{prop}

\begin{proof} [Proof of Theorem \ref{refined} assuming Proposition \ref{bound}]

The case $\Sigma\cdot \Sigma \leq |\langle c, [\Sigma]\rangle|$ has been dealt with in Section~\ref{3.2}. 

So we assume that $\Sigma\cdot \Sigma > |\langle c, [\Sigma]\rangle|$ as in Proposition \ref{bound}. 
If we also assume that $\Sigma$ has positive genus, 
then Theorem \ref{refined} follows from 
Proposition~\ref{bound}.

So to prove Theorem \ref{refined}, 
it remains to show that $\Sigma$ cannot be a sphere if  $\Sigma\cdot \Sigma-|\langle c, [\Sigma]\rangle|\geq 2$. 

Assume $\Sigma$ is a sphere and satisfies $\Sigma\cdot \Sigma-|\langle c, [\Sigma]\rangle|\geq 2$. We can attach a trivial  handle to $\Sigma$ and construct a torus $T$, which is homologous to $\Sigma$. Observe that  $2g(T)-2=0$, which is less than $T \cdot T-|\langle c, [T]\rangle |\geq 2$.  
So this is impossible  by Proposition \ref{bound}.  

Therefore we do obtain the  adjunction inequality.
\end{proof}

\subsection{Proof of Proposition \ref{bound}}

We start with the following inequality in the case $\Sigma\cdot \Sigma=0$, which
is  essentially the combination of Proposition 8, Lemma 9 and Lemma 10  of  \cite{KM}. 
The arguments are the same with the ordinary SW invariants replaced by the $V(X)$-valued invariants {\bf SW}.

\begin{lemma}\label{km}
Let $X$ be a smooth, closed, connected, oriented 4-manifold with $b^+=1$, and $\Sigma \subset X$ a connected smooth embedded surface with $\Sigma\cdot \Sigma=0$ and 
$g(\Sigma)>0$.
Let $\mathcal L$ be a Spin$^c$ structure and use the forward cone $\mathcal C$ of $[\Sigma]$ to define ${\bf SW}_\pm (X,\mathcal{L})$.

If  $\langle c_1(\mathcal{L}), [\Sigma]\rangle >0$ and ${\bf SW}_+(X,\mathcal{L})\neq 0$,then 
$$2g(\Sigma)-2\geq |\langle c_1(\mathcal{L}), [\Sigma]\rangle|.$$
 The inequality also holds if  $\langle c_1(\mathcal{L}), [\Sigma]\rangle <0$ and ${\bf SW}_-(X,\mathcal{L})\neq 0$. 
\end{lemma}

\begin{proof}
Assume $\langle c_1(\mathcal{L}), [\Sigma]\rangle >0$ and ${\bf SW}_+(X,\mathcal{L})\neq 0$.
Assume $H$ is an integral class in $\mathcal{C}$ with $H^2=1$ and $U$ is a tubular neighborhood of $\Sigma$ diffeomorphic to $\Sigma \times D$, where $D$ is a $2$-dimensional disk. Let $Y=\p U$, which is diffeomorphic to $\Sigma\times S^1$. Choose a metric
$g$ on $X$ such that it has product form $g_\Sigma +d\theta^2 +dt^2$ in a collar $Y\times (-\ep,\ep)$, where $g_\Sigma$ is a metric on $\Sigma$ with constant scalar curvature $-4\pi(2g-2)$ and unit area, $d\theta^2$ is the metric on $S^1$,  and $dt^2$ is the metric on $(-\ep,\ep)$.

Let $g_R$ be the metric given by attaching a cylinder $Y\times [-R,R]$ with product metric $g_\Sigma +d\theta^2 +dt^2$ between $X-U$ and $U$. Let $\omega_{g_R}$ be the unique self-dual harmonic form in $\mathcal{C}$ normalized by $H\cdot [\omega_{g_R}]=1$. Take a sequence $\{R_i\}$ increasing to infinity. 
The corresponding sequence of self-dual harmonic forms $\{\omega_{g_{R_i}}\}$ has a subsequence (also called $\{\omega_{g_{R_i}}\}$) converging to a limit, called $\omega_{LN}$.
By the proof of Lemma 10 in \cite{KM}, $\langle [\omega_{g_{R_i}}], [\Sigma]\rangle \rightarrow 0$.
By the light cone lemma (\cite{LiLiu2} Lemma 2.6), $[\omega_{LN}]= \frac{1}{\langle H, [\Sigma]\rangle} \text{PD}([\Sigma])$.
In particular, $c_1(\mathcal{L})\cdot [\omega_{g_{R_i}}]>0$ as $i>>0$.
Since ${\bf SW}_+(X,\mathcal{L})\neq 0$, $\mathcal{M}_{X,g_{R_i}}(\mathcal{L})$ is non-empty when $i>>0$.
By \cite{KM} Proposition 8 and Lemma 9, $2g(\Sigma)-2\geq |\langle c_1(\mathcal{L}), [\Sigma]\rangle|$ as required.

The proof for $\langle c_1(\mathcal{L}), [\Sigma]\rangle <0$ and ${\bf SW}_-(X,\mathcal{L})\neq 0$ case is the same.
\end{proof}

Under the conditions of this lemma, we call $g_{R_i}$ the long neck metric. We also call the chamber containing $\omega_{R_i}$ with $i>>0$ the long neck chamber and define ${\bf SW}_{LN}(X,\mathcal{L})$ as the {\bf SW} invariant of $X$ with respect to $\mathcal{L}$ and the chamber.

\begin{proof}[Proof of Proposition \ref{bound}]

If we blow up $X$ at $n=\Sigma\cdot \Sigma$ points on $\Sigma$, the blowup surface $\tilde{\Sigma}\subset X\sharp n\overline{\C\PP}^2$ satisfies
$$g(\tilde{\Sigma})=g(\Sigma), \quad [\tilde{\Sigma}]=[\Sigma]-\sum \text{PD}(E_i), \quad  \tilde{\Sigma}\cdot\tilde{\Sigma}=0,$$ with $E_1,\cdots ,E_n$ the exceptional classes.
Suppose  $\mathcal{L}$ is a spin$^c$ structure such that its SW dimension $k(\mathcal L)\geq 0$, and the wall crossing value ${\bf SW}_+(X,\mathcal{L})-
{\bf SW}_-(X,\mathcal{L})$ is nonzero (in $\wedge^* H^1(X;\Z)/Tor$),  and 
$|\langle c_1(\mathcal{L}), [\Sigma]\rangle |< \Sigma\cdot\Sigma.$

Notice that in the formulae of Lemma \ref{formulae},
the forward cone $\tilde{\mathcal{C}}$ of $X\sharp n\overline{\C\PP}^2$ is chosen such that $\mathcal{C}\subset \tilde{\mathcal{C}}$.
Hence PD$([\tilde{\Sigma}])\in \tilde{\mathcal{C}}$.

Let $\mathcal L_1, \mathcal L_2$ be the spin$^c$ structures on  $X\sharp n\overline{\C\PP}^2$ with 
 $$c_1(\mathcal{L}_1)=c_1(\mathcal{L})+\sum E_i, $$
$$c_1(\mathcal{L}_2)=c_1(\mathcal{L})-\sum E_i.$$
Then $k(\mathcal{L}_1)=k(\mathcal{L}_2)=k(\mathcal{L})\geq 0$ and
$$\langle c_1(\mathcal{L}_1), [\tilde{\Sigma}]\rangle =\langle c_1(\mathcal{L}), [\Sigma]\rangle +\Sigma\cdot\Sigma>0, $$
$$\langle c_1(\mathcal{L}_2), [\tilde{\Sigma}]\rangle =\langle c_1(\mathcal{L}), [\Sigma]\rangle -\Sigma\cdot\Sigma<0.$$
So PD$([\tilde{\Sigma}])$ is in the positive (resp. negative) chamber of $c_1(\mathcal{L}_1)$ (resp. $c_1(\mathcal{L}_2)$).

Consider the long neck metric $\tilde{g}_{R_i}$ on $X\sharp n\overline{\C\PP}^2$ related to $\tilde{\Sigma}$, we know that the class of limit 2-form is $[\tilde{\omega}_{LN}]=\frac{1}{\langle H,[\tilde{\Sigma}]\rangle} \text{PD}([\tilde{\Sigma}])$.
So $\tilde{\omega}_{LN}$ is in the positive (resp. negative) chamber of $c_1(\mathcal{L}_1)$ (resp. $c_1(\mathcal{L}_2)$) as well.

By the blowup formula (Lemma \ref{formulae}), we have
$${\bf SW}_{\pm}(X\sharp n\overline{\C\PP}^2, \mathcal{L}_i) ={\bf SW}_{\pm}(X,\mathcal{L}), \quad  i=1, 2.$$ 
By the assumption of non-trivial wall crossing, 
at least one of $${\bf SW}_{\pm}(X, \mathcal{L})$$  is nonzero.
Hence either $${\bf SW}_+(X\sharp n\overline{\C\PP}^2, \mathcal{L}_1) \neq 0$$
 or 
 $${\bf SW}_-(X\sharp n\overline{\C\PP}^2, \mathcal{L}_2)\neq 0.$$
By Lemma \ref{km},
$$2g(\Sigma)-2=2g(\tilde{\Sigma})-2\geq | \langle c_1(\mathcal{L}_i), [\tilde{\Sigma}]\rangle |=|\langle c_1(\mathcal{L}), [\Sigma]\rangle \pm \Sigma\cdot\Sigma|.$$
In either case, $2g(\Sigma)-2\geq \Sigma\cdot\Sigma-|\langle c_1(\mathcal{L}), [\Sigma]\rangle |$.
\end{proof}


\section{Applications}

We will apply Theorem~\ref{main} to obtain explicit genus bound for cohomology classes with non-negative square.

\subsection{Proof of Theorem \ref{optimal}}
Let us make a simple observation.

\begin{lemma}\label{sum with S1xS3}
Suppose $Y$ is a closed 4-manifold and $X=Y\#(S^1\times S^3)$. Then under the natural isomorphism $H_2(Y;\Z)\cong H_2(X;\Z)$, we have
$$mg_X(A)\leq mg_Y(A)\quad \hbox{and}  \quad h_{\Lambda(X)}(A)=h_{\Lambda(Y)}(A).$$
\end{lemma}
\begin{proof}
The inequality about $mg$  is clear since any smooth surface in $Y$ can be considered as a surface in $X$ 
representing the same class. 

The equality about $h$ follows from Lemma \ref{reduction}.
\end{proof}

\begin{prop} \label{symplectic}
Suppose that $\Lambda$ is a cohomology algebra of $b^+=1$ type. When $2\tilde\chi(\Lambda)+3\sigma(\Lambda)\geq 0$ and $\Lambda$ is Lefschetz, there exists a closed, connected, symplectic 4-manifold $X_{\Lambda}$ 
and an algebra  isomorphism $H^*(X_{\Lambda};\Z)/Tor\cong \Lambda$ such that 
$$h(A)=mg_{X_{\Lambda}}(A)$$
 whenever $h(A)\geq 0$.  
Furthermore,  $X_{\Lambda}$ can be chosen   from the list in Theorem \ref{optimal}. 
\end{prop}

\begin{proof} We still divide the discussion into five cases as in the beginning of Section 2.3
and choose $X_{\Lambda}$ in each case.

{\bf Case (1)}  $T$ trivial and  $\Gamma=U$.  In this case, let $X_{\Lambda}=S^2\times S^2$.

{\bf Case (2)} $T$ trivial and $\Gamma=U\oplus (-E)$.  In this case, choose $X_{\Lambda}$ to be the  Enriques surface.

{\bf Case (3)} $\tilde b_1(\Lambda)= b_1(\Lambda)=2$ and $\Gamma=U$. In this case, let   $X_{\Lambda}=S^2\times T^2$. 
 
{\bf Case (4)} $T$ trivial and $\Gamma=\langle 1 \rangle \oplus n\langle -1\rangle, 0\leq n\leq 9$.
  In this case, let $X_{\Lambda}=\C\PP^2 \sharp n \overline{\C\PP}^2$.

{\bf Case (5)} $\tilde b_1(\Lambda)= b_1(\Lambda)=2$ and $\Gamma=\langle 1\rangle \oplus \langle -1\rangle$. In this case, choose  $X_{\Lambda}$ to be the non-trivial $S^2$-bundle over $T^2$. 

We will need the following facts about these symplectic manifolds $X_{\Lambda}$. 

\begin{lemma} \label{properties}
For those symplectic manifolds $X=X_{\Lambda}$,
\begin{itemize}
\item $mg_X(A)$ has been computed for any class $A$ with $A\cdot A\geq 0$.

\item  $D(X)=Aut(\Gamma)^T$, where $D(X)$ is the subgroup of $Aut(\Gamma)$ induced by orientation-preserving diffeomorphisms of $X$.

\end{itemize}
\end{lemma}

\begin{proof}
The fact that $D(X)=Aut(\Gamma)^T$ for $S^2\times S^2$,  $S^2$-bundles over $T^2$ is 
straightforward (see e.g \cite{LL1} Remark 3). The equality for $\C\PP^2 \sharp n \overline{\C\PP}^2, n\leq 9$
 was due to Wall \cite{Wall} Theorem 2 and the following corollary, and for Enriques surface due to L\"onne \cite{Lonne} Theorem 8 (extending  Friedman-Morgan \cite{FM}).

  The minimal genus function $mg$  was computed in \cite{LL}, \cite{Rub} for $ S^2\times S^2$,  in 
\cite{LL1}
 for $S^2$-bundles over $ T^2$, and in \cite{LL}, \cite {Li} for $\C\PP^2 \sharp n \overline{\C\PP}^2, n\leq 9$. 
 For the Enriques surface, the generalized Thom conjecture in \cite{MST96} implies that  the genus bound \eqref{enriques}  is valid for any class with non-negative square. Moreover,   the minimal genus of $A$ with $A\cdot A\geq 0$ is  given precisely by  $\frac{1}{2}  (A\cdot A +2)$,
which  can be seen from symplectic Seiberg-Witten theory as in \cite{LiLiu4}. 

\end{proof}

By Lemma \ref{properties},  for any class of non-negative square $A$ of  $X_{\Lambda}$, there is a geometric automorphism $\phi\in D(X_{\Lambda})$
such that $A'=\phi(A)$ is reduced. 
Since $mg_X(A')=mg_X(A)$ and $h(A')=h(A)$,
it suffices to show that $h=mg_X$ for reduced classes with non-negative square and non-negative $h$. 

This equality is clear from comparing 
\eqref{ha-1}-\eqref{ha-5},  \eqref{negoddgenus} and the references in Lemma \ref{properties}.



\end{proof}

We remark that there are actually connected symplectic surfaces in the class $A$ for some symplectic forms,  
whenever $A\cdot A>0$, or
$A\cdot A=0$ and $A$ is primitive. 
Another fact is that  when the symplectic form is reduced, the symplectic canonical class is the same as $c_0$.

Now we complete the proof of Theorem \ref{optimal}. 

\begin{proof} [Proof of Theorem \ref{optimal}]

Let $\Lambda_{red}$  be the reduction of $\Lambda$. 
Then 
$$2\tilde\chi(\Lambda_{red})+3\sigma(\Lambda_{red})=2\tilde\chi(\Lambda)+3\sigma(\Lambda)\geq 0.$$
For $A\in \Lambda^2$ with $A\cdot A>0$, under the natural isomorphism between $\Lambda^2$ and $\Lambda_{red}^2$, Lemma \ref{reduction}, Corollary \ref{leq0} and Lemma \ref{leq0case4} imply $h(A)\geq 0$.

Let $Y=X_{\Lambda_{red}}$ as in Proposition \ref{symplectic}. Then under an isomorphism 
(in fact  under any isomorphism)  $H^*(Y;\Z)/Tor\cong \Lambda_{red}$, we have 
$$h(A)=mg_{Y}(A)$$
 whenever $h(A)\geq 0$.

 Let  $X_{\Lambda}=Y\# l (S^1\times S^3)$,  where $l=b_1(\Lambda)-\tilde b_1(\Lambda)$.
 By Theorem \ref{main}, Lemma \ref{sum with S1xS3} and Proposition \ref{symplectic}, we have 
 $$h_{\Lambda(X_{\Lambda})}(A)\leq   mg_{X_{\Lambda}}(A)\leq  mg_Y(A)=h_{\Lambda(Y)}(A)=h_{\Lambda(X_{\Lambda})}(A),$$
 whenever $h(A)\geq 0$. 
 
 This implies the desired equality.

\end{proof}

\subsection{Sphere representability}
In light of Theorem \ref{main} and Theorem \ref{optimal},  when $2\tilde \chi+3\sigma\geq 0$,
Corollary \ref{leq0} and Lemma \ref{leq0case4} provide strong  
constraints for classes represented by embedded spheres.  We summarize the constraints
as follows. 

\begin{prop} \label{sphere}
Suppose $2\tilde \chi(\Lambda)+3\sigma(\Lambda)\geq 0$ and $A\in \Lambda^2$ is a nonzero reduced class with $A\cdot A\geq 0$
and represented by a sphere. 

If $A\cdot A>0$, then $\tilde b_1(\Lambda)=0$. And either
\begin{itemize}
\item $\Lambda$ is in {\bf Case (1)} and $A=aF+B$ with $a\geq 1$, or
\item $\Lambda$ is in {\bf Case (4)} and $A$ is one of the following classes
$$H,\quad 2H,    \quad       a H - (a-1) E_1,  \quad    aH- (a-1) E_1 -E_2,  \quad  a\geq 2.$$

\end{itemize}

If $A\cdot A=0$, then either
\begin{itemize}
\item $\Lambda$ is in {\bf Cases (1), (3), (5)}  and $A=aF$ with $a\geq 1$, or
\item $\Lambda$ is in {\bf Case (4)} and $A=a(H-E_1)$ with $a\geq 1$.

\end{itemize}
\end{prop}

When $2\tilde \chi+3\sigma<0$, the calculation of $h$ is harder.  Nevertheless, we still hope to show:
the only  possible  classes (with non-negative square)  represented by spheres are 
the same as in the case   $2\tilde \chi+3\sigma\geq 0$.

In the 80s and 90s   using Donaldson's diagonalization of definite smooth manifolds, there have been many applications to representing classes by spheres in manifolds with $b^+=1$. 
The most general one is due to Kikuchi (Theorem 1 in \cite{Ki}), which is reformulated for the trivial 
$T$ case as follows:

\begin{theorem}  \label{T trivial}
 Suppose $X$ is a smooth, closed, oriented 4-manifold with trivial $T(X)$.  Suppose $A$ is a class with $A\cdot A>0$ and represented by an embedded sphere. Then $A$ is equivalent under $Aut(\Gamma)$ to one of the classes in the first part of Proposition \ref{sphere}.
 


\end{theorem}

When $T$ is non-trivial, we have the following  simple observation. 

\begin{lemma} \label{T non-trivial}
When $T$ is non-trivial,  there are no classes with positive square and representable by spheres.
\end{lemma}

\begin{proof}
Suppose $A$ is a class with $A\cdot A=s>0$ and represented by an embedded sphere $S$ in $X$. Let $(X', S')$ be the connected sum of the pair $(X, S)$ and $s-1$ copies of the pairs $(\overline{\C \PP^2}, \overline{\C \PP^1})$. Notice that $S'$ is a sphere with self-intersection $1$
whose neighborhood has $S^3$ as boundary.
Thus $(X', S')$ can be decomposed into the connected sum $(Z, \emptyset)\# (\C \PP^2, \C \PP^1)$, 
where $Z$ has negative-definite intersection form. However, $H^1(Z;\Z)=H^1(M;\Z)$ and hence
$F$ is in $H^2(Z;\Z)$ and non-trivial. But this is impossible since $F\cdot F=0$ and $Z$ is negative definite.
\end{proof}

So the remaining case is 
$A\cdot A=0.$

\begin{conj}\label{square zero} Suppose $2\tilde \chi+3\sigma<0$. 
Suppose $A$ is a class with $A\cdot A=0$ and represented by an embedded sphere.
Then $A$ is equivalent under $Aut(\Gamma)^T$ to one of the classes in the last part of Proposition \ref{sphere}.


\end{conj}

\begin{remark}
For rational or ruled 4-manifolds, all the classes in 
Conjecture \ref{square zero} and Theorem \ref{T trivial}
are indeed represented by embedded spheres. 
\end{remark}

So far we can verify Conjecture~\ref{square zero} for the following three cases. 

\vspace{3mm}
\noindent\underline {Case 1}. $\tilde b_1\geq 4$, $\Gamma=U$ as {\bf Case (3)} in Section \ref{nonneg inv}. The only difference is $2\tilde \chi+3\sigma=8-4\tilde b_1<0$. $A=aF+bB$ is reduced with non-negative square if $a\geq 0$, $b\geq 0$.  Again $c_0=2B$ is a type II adjunction class. The only possible spherical classes are $aF$ according to formula \eqref{ha-3}.

\noindent \underline{Case 2}. $\tilde b_1\geq 4$, $\Gamma=V$ as {\bf Case (5)} in Section \ref{nonneg inv}.  $A=aF+bB$ is reduced with non-negative square if $a\geq 0$, $b\geq 0$, or $a>0$, $b\leq -2a$.  $c_0=2B-F$ is a type II adjunction class. The only possible spherical classes are $aF$ according to formula \eqref{ha-5}.

\noindent \underline{Case 3}. $\tilde b_1\geq 2$, $\Gamma=U\oplus \langle -1 \rangle$. 
The verification in this  case is elementary but a bit long, so we omit the details.

The remaining cases for Conjecture~\ref{square zero} are:
$T$ trivial, $\Gamma=\langle 1\rangle \oplus n \langle -1\rangle$ with $n\geq 10$, and $\tilde b_1\geq 2$, $\Gamma=U\oplus (n-1)\langle -1\rangle$ with $n\geq 3$.

 \appendix
 \section{Proofs of Lemmas \ref{case4} and \ref{leq0case4}}
 
 \begin{proof} [Proof of Lemma  \ref {case4}]
 Recall that  $A=aH-\sum_{i=1}^n b_i E_i$ is reduced means that  $b_1\geq b_2\geq \cdots \geq b_n\geq 0$, and  $a\geq b_1+b_2+b_3$.  We also have $A\cdot A=a^2 -\sum_{i=1}^n b_i^2 \geq 0$.

Let $ c=kH-\sum_{i=1}^ne_iE_i$ be a characteristic class. Then 
$k$ and $e_i$ are odd. 
We may assume $k>0$.
Since $T$ is trivial, there are only  type I adjunction classes. Such a class $c$ satisfies
$$c\cdot c=k^2-\sum e_i^2\geq \sigma+8=9-n\geq 0. $$
Since $e_i\ne 0$ for any $i$, we must have $k\geq 3$. 
Let $c_0=3H-\sum_{i=1}^nE_i$. Then 
$$c_0\cdot A=3a-\sum b_i\geq 3(b_1+b_2+b_3)-\sum b_i\geq 0.$$
So $|c_0\cdot A|=c_0\cdot A$ and we obtain the formula for $h_{c_0}$ as in 
\eqref{negoddgenus}. 

We will  show $h(A)=h_{c_0}(A)$ by verifying that $c\cdot A\geq c_0\cdot A$ among all type I adjunction classes with positive $k$ coefficient. That is, we will prove that
\be
ka-\sum_{i=1}^n e_i b_i \geq 3a -\sum_{i=1}^n b_i ,
\label{ke1}
\ee
where $(k; e_1,\cdots, e_n)$ and $(a; b_1,\cdots, b_n)$ satisfy the said conditions. 
Observe that 
$$ ka-\sum_{i=1}^n e_i b_i \geq ka -\sum_{i=1}^n |e_{\tau(i)}| \, b_i$$
by a rearrangement, where $\tau$ is a permutation of $\{ 1,2,\cdots, n\}$ such that $|e_{\tau(1)}|\geq |e_{\tau(2)}|\geq \cdots \geq |e_{\tau(n)}| \geq 1$. So it suffices to prove \eqref{ke1} under an extra condition 
$$ e_1 \geq e_2 \geq \cdots \geq e_n \geq 1.$$

Since $a\geq 1$ for nonzero reduced classes, let $x_i=b_i/a$ for all $i$.  Then 
$$ (k-3)a -\sum_{i=1}^n (e_i -1) b_i =a[k-3 -\sum_{i=1}^n (e_i-1)x_i] .$$
If  $n<9$, we may extend to $n=9$ case by setting $e_{n+1}=\cdots =e_9 =1$ and $x_{n+1} =\cdots =x_9=0$. So we need to prove that 
\be
k-3 -\sum_{i=1}^9 (e_i -1) x_i \geq 0
\label{ke2}
\ee
under conditions
\be
\begin{cases}
k>e_1\geq e_2 \geq \cdots \geq e_9 \geq 1, & \text{all odd,} \\
k^2 -\sum_{i=1}^9 e_i^2 \geq 0, & 
\end{cases}
\label{ke3}
\ee
and 
\be
\begin{cases}
\sum_{i=1}^9 x_i^2 \leq 1, & \\
x_1 \geq x_2 \geq \cdots \geq x_9 \geq 0, & \\
x_1 +x_2 +x_3 \leq 1. & 
\end{cases}
\label{ke4}
\ee

Notice that the second and third inequalities in \eqref{ke4} imply that
\begin{align*} 
1 &\geq (x_1+x_2 +x_3)^2 
 \geq \sum_{i=1}^9 x_i^2. 
\end{align*}
So the first inequality in \eqref{ke4} is redundant and will be ignored. 

If $e_1=e_2=\cdots =e_9=1$, then $k\geq 3$, and inequality \eqref{ke1} or \eqref{ke2} holds obviously.
In any other cases, $e_1\geq 3$. 
For fixed $(k; e_1,\cdots, e_9)$ satisfying \eqref{ke3} and $e_1\geq 3$, we will try to show that 
$$ \sum_{i=1}^9 (e_i -1) x_i\leq k-3 $$
for any $x_1, x_2, \cdots, x_9$ satisfying constraints \eqref{ke4}. It is clear that
$$\sum_{i=1}^9 (e_i -1) x_i  \leq (e_1 -1)x_1 +(e_2-1)x_2 +\left( 
\sum_{i=3}^9 (e_i -1)\right)x_3=:P(x_1,x_2,x_3).$$
Hence it is enough to find the maximum value of $P(x_1,x_2,x_3)$ when $x_1+x_2+x_3\leq 1, x_1\geq x_2\geq x_3\geq 0$, and show that it is not greater than $k-3$. 
Since $1-x_2-x_3\geq x_1$, we have
$P(x_1,x_2,x_3)\leq P(1-x_2-x_3,x_2,x_3)$ and can assume $x_1+x_2+x_3=1$.

If $\sum_{i=3}^9 (e_i -1)\leq e_1 -1$, by the condition $k\geq e_1 +2$,
$$P(x_1,x_2,x_3)\leq (e_1-1)x_1+(e_1-1)x_2+(e_1-1)x_3=e_1-1\leq k-3.$$

Assume $\sum_{i=3}^9 (e_i -1) >e_1 -1\geq e_2 -1$ now. Note that 
$$P(x_1,x_2,x_3)\leq P\left(x_1,\frac{x_2+x_3}{2},\frac{x_2+x_3}{2}\right).$$ 
So we can assume further that $x_2=x_3$. Then $x_2=\frac{1-x_1}{2}$ and
\begin{align*} P(x_1,x_2,x_2)  &=(e_1 -1)x_1 +\left(\sum_{i=2}^9 (e_i -1)\right)\frac{1-x_1}{2}\notag \\
&=\frac12  \sum_{i=2}^9 (e_i -1) +\frac{x_1}{2} \left[2(e_1 -1) - \sum_{i=2}^9 (e_i -1)\right]. 
\end{align*}
If $2(e_1 -1) \geq \sum_{i=2}^9 (e_i -1)$, $P(x_1,x_2,x_2)$ attains maximum value at $x_1=1$ and 
$P(1,0,0)=(e_1-1)\leq k-3$.

If $2(e_1 -1) < \sum_{i=2}^9 (e_i -1)$, $P(x_1,x_2,x_2)$ attains maximum value at $x_1=x_2=\frac13$, and  we have 
$$P\left(\frac13,\frac13,\frac13\right)=\frac13  \sum_{i=1}^9 (e_i -1)=\left(\frac13  \sum_{i=1}^9e_i\right)-3\leq k-3$$ 
where the last inequality follows $k^2 \geq \sum_{i=1}^9 e_i^2$ and Cauchy-Schwarz inequality. 

This completes the proof of inequality \eqref{ke1}, hence the Lemma. 
\end{proof}

\begin{proof}[Proof of Lemma \ref{leq0case4}]

By Lemma \ref{case4}
we need to solve 
$$ 2h_{}(A) =a^2-3a+2-\sum_{i=1}^{n} b_i(b_i -1) \leq 0 $$
under constraints:
\begin{enumerate}

\item[(1)] $A^2=a^2 -\sum_{i=1}^n b_i^2 \geq 0$;
 
\item[(2)] $b_1\geq b_2 \geq \cdots \geq b_n \geq 0$;

\item[(3)] $a\geq b_1 +b_2 +b_3$;

\item[(4)] $1\leq n\leq 9$.

\end{enumerate}
As noted in the proof of Lemma \ref{case4}, the first inequality is redundant and will be ignored. 
We follow the method in Section 5 of \cite{Li}. Let $a=b_1 +t$. Constraints (2) and (3) imply that $t\geq b_2 +b_3 \geq 2b_3$. 
If $t=0$, then $b_2=b_3=\cdots =b_n=0$, and $a=b_1\geq 1$. So $A=a(H-E_1)$, and $2h_{}(A) =(a-1)(a-2) -a(a-1) =-2(a-1)\leq 0$. 

If $t=1$, then $b_2+b_3 \leq 1$. So either 
$b_2=1, b_3=\cdots =b_n=0$, or $b_2=b_3=\cdots =b_n=0$. In either case, we have $h_{}(A)=0$. 

Now assume $t\geq 2$. 
We want to show that min$h(A)=0$ in this case.
Substitute $b_1=a-t$, and use constraints (2) and (4) to get
\begin{align*}
2h(A)&= a^2-3a+2-(a-t)^2 +(a-t) -\sum_{i=2}^{n} b_i(b_i -1) \\
&= 2(t-1)a -t^2-t+2 -b_2^2 +b_2 -\sum_{i=3}^{n} (b_i^2 -b_i) \\ 
&\geq 2(t-1) (b_2 +t) -t^2-t+2 -b_2^2 +b_2 -  7 (b_3^2 -b_3) \\
&:= q(b_2) ,
\end{align*} 
where in the definition of $q(b_2)$, we have fixed $t$ and $b_3$.  So $b_2\in [b_3, t-b_3]$. It is clear that $q$ is a quadratic function in $b_2$, so it takes minimum value only at end points of the interval $[b_3, t-b_3]$. But
$$ q(t-b_3) -q(b_3) =(t-1)(t-2b_3) \geq 0.$$ 
It follows that
$$ 2h(A) \geq q(b_3) =2(t-1) (b_3+t)-t^2-t+2 -8(b_3^2 -b_3) :=l_t(b_3),$$ 
where $0\leq b_3 \leq \frac{t}{2}$. For fixed $t$, $l_t$ is a quadratic function in $b_3$, and takes minimum value only at end points of the interval $[0, \frac{t}{2}]$. 
Since $t\geq 2$,
$$ l_t(0) =2(t-1)t -t^2 -t+2 =t^2-3t+2 =(t-1)(t-2) \geq 0$$
and 
$$
l_t(\frac{t}{2}) = 2(t-1) \cdot \frac32 t-t^2-t+2 -8\cdot \frac{t}{2} \cdot \frac{t-2}{2} =2>0.
$$ It follows that
$2h(A) \geq l_t(b_3) \geq 0$, and equality holds only if $t=2$ and $b_2=b_3=0$. So the only possible solutions for $h(A)\leq 0$ in this case are $A=(b+2) H -bE_1$, $b\geq 0$. Compute 
$$ 2h(A)=(b+1)b-b(b-1) =2b\leq 0  $$
to see that the only solution is $b=0$, i.e. $A=2H$. 

Combining the above discussion, we see that all possible solutions to $h(A)\leq 0$ for case (4) are listed in the statement. 
\end{proof}

\end{document}